\documentclass{scrartcl}
\usepackage[margin=2cm,bottom=1cm,includefoot]{geometry}
\usepackage{amsmath,amssymb,amsthm}
\usepackage{bbm}
\usepackage{array}
\usepackage{mathrsfs}
\usepackage{enumerate}
\usepackage{graphicx}
\usepackage[utf8]{inputenc}
\usepackage{hyperref}
\usepackage[all]{hypcap}
\usepackage{float}

\newtheorem{thm}{Theorem}
\newtheorem{lem}[thm]{Lemma}
\newtheorem{prop}[thm]{Proposition}
\newtheorem{cor}[thm]{Corollary}
\theoremstyle{remark}
\newtheorem{rem}[thm]{Remark}

\let\BFseries\bfseries\def\bfseries{\BFseries\mathversion{bold}}

\newcommand{\N}{\mathbb{N}}
\newcommand{\Z}{\mathbb{Z}}
\newcommand{\Q}{\mathbb{Q}}
\newcommand{\R}{\mathbb{R}}
\newcommand{\C}{\mathbb{C}}
\newcommand{\E}{\mathbb{E}}
\newcommand{\p}{\mathbb{P}}
\newcommand{\CC}{\mathcal{C}}
\newcommand{\NN}{\mathcal{N}}
\newcommand{\BB}{\mathcal{B}}
\newcommand{\LL}{\mathcal{L}}

\newcommand{\1}{\mathbbm{1}}
\newcommand{\eps}{\varepsilon}
\DeclareMathOperator{\e}{e}
\newcommand{\D}{\mathrm{d}}
\newcommand{\op}{\operatorname}
\newcommand{\lne}{<}
\newcommand{\gne}{>}
\newcommand{\abs}{|}

\begin{document}
\title{Brownian motion conditioned to spend limited time outside a bounded interval -- an extreme example of entropic repulsion}
\author{Frank Aurzada\footnote{Technical University of Darmstadt}\and Martin Kolb\footnote{Paderborn University} \and Dominic T. Schickentanz\footnotemark[1]$^{\ ,}$\footnotemark[2]}
\date{\today}
\maketitle
\begin{abstract}
We show that a Brownian motion on $\R_{\geq 0}$ which is allowed to spend a total of $s \gne 0$ time units outside a bounded interval does not leave the interval at all. This can be seen as an extreme example of entropic repulsion. Moreover, we explicitly determine the exact asymptotic behavior of the probability that a Brownian motion on $[0,T]$ spends limited time outside a bounded interval, as $T\to\infty$.
\end{abstract}
\section{Introduction and main results}
This paper is concerned with stochastic processes under constraints. Starting with Doob in~\cite{Doo57}, where Brownian motion is conditioned not to enter the negative half-line, many other constraints have been studied. Here, we continue a line of research which was initiated by Benjamini and Berestycki in~\cite{BB10} and~\cite{BB11}. In both papers, a Brownian motion is forced to fulfil a condition that is atypical for the process and it turns out that, due to entropic repulsion effects, the conditioned process happens to fulfill an even more stringent condition than required.
\\
More precisely, in~\cite{BB11}, the local time of Brownian motion at the origin is required to stay below a given deterministic non-decreasing function. If, in particular, this function is constant, the conditioned process does not exhaust the full allowed local time. Extensions are given by~\cite{KS16} and~\cite{Bar20}. Further, in~\cite{BB11}, the case where the time a standard Brownian motion spends in the negative half-line has to stay bounded by a constant is also treated. Again, the conditioned process does not exhaust the full allowed occupation time.
This result has been extended by~\cite{AS22}.
In~\cite{BB10}, the authors consider the case where the Brownian local times at all places have to stay bounded by a common constant. In all these situations, effects of entropic repulsion are clearly visible.
\\
In the present paper, we will encounter a very rare extreme example of entropic repulsion: If Brownian motion is forced to spend \textit{only limited time} outside a bounded interval, then the resulting process does not spend \textit{any time at all} outside the interval. In other words, the resulting process does not make any use of the possibility to leave the interval, which is somewhat surprising. Besides a rather involved conditioning considered in~\cite{Pro15}, it is, in fact, the first and only time so far such an extreme behavior has been discovered.
\\
In order to explain our result in precise terms, let us first introduce some basic notation.
Fix~$s \gne 0$, let~$B=(B_t)_{t \ge 0}$ be a Brownian motion starting in~$y \in (-1,1)$ and let
$$
\Gamma_T:= \int_0^T \1_{\{\abs B_t\abs \ge 1\}} \D t, \quad T \ge 0,
$$
be the time~$B$ spends outside the interval~$(-1,1)$ until time~$T$. By the scaling property and translation invariance of Brownian motion, one can immediately transfer our results to the case of an arbitrary bounded interval.

Our main result can be stated as follows:
\begin{thm}
\label{thm:main}
As $T \to \infty$, the probability measures $\p_y(B \in \cdot \,\, \abs\, \Gamma_T \le s)$ and $\p_y(B \in \cdot \,\, \abs\, \Gamma_T = 0)$ converge weakly to the same limit on the Wiener space~$\mathcal{C}([0,\infty),\mathbb{R})$. The limiting process~$(Y_t)_{t \ge 0}$ satisfies the SDE
\begin{equation}
\label{eqn:sde}
Y_0=y,\quad \D Y_t = \D W_t - \frac{\pi}{2} \tan\left(\frac{\pi Y_t}{2}\right) \D t, \quad t \ge 0,
\end{equation}
where~$(W_t)_{t \ge 0}$ is a standard Brownian motion.
\end{thm}
This result shows that a Brownian motion which is forced to spend \textit{less than~$s$ time units outside a bounded interval} will end up \textit{not leaving the interval at all}, as announced above. In fact, \eqref{eqn:sde} is precisely the SDE of the taboo process considered in~\cite{Kni69}, i.e., of a Brownian motion conditioned to stay in the interval.
\\
Our methodology also allows to exhibit the exact asymptotic behavior of~$\p_y(\Gamma_T \le s)$, as~$T \to \infty$, explicitly:
\begin{prop}
\label{prop:asymp}
As~$T \to \infty$, we have
$$\p_y(\Gamma_T \le s) \sim \frac{\cos\left(\frac{\pi y}{2}\right) 2^{\frac{19}{6}}}{\sqrt{3} \pi^{\frac{13}{6}} s^{\frac{1}{6}} T^{\frac{1}{3}}} \exp\left( -\frac{\pi^2}{8} T + \frac{3}{2^{\frac{7}{3}}} \pi^{\frac{4}{3}} s^{\frac{1}{3}} T^{\frac{2}{3}} - \frac{3}{2^{\frac{5}{3}}} \pi^{\frac{2}{3}} s^{\frac{2}{3}} T^{\frac{1}{3}} +\frac{\pi^2+12}{24}s \right), \quad y \in (-1,1),$$
as well as
\begin{equation}
\label{eqn:outsideasymptotics}   
\lim_{T\to\infty} \frac{\p_x(\Gamma_T \le s)}{\p_y(\Gamma_T \le s)} = 0,\quad x\in \R \setminus (-1,1),\ y\in(-1,1).
\end{equation}
\end{prop}
This asymptotic behavior is of interest in its own right due to the explicitness of the unusual polynomial and subexponential terms as well as the discontinuity w.r.t.~the starting point implied by~\ref{eqn:outsideasymptotics}. It should be compared to the probability that~$B$ never leaves~$(-1,1)$: As~$T \to \infty$, the classical formula for first exit probabilities (see, e.g., Example~5a in~\cite{DS53}) yields
\begin{align}
\label{eqn:asympGamma0}
\p_y(\Gamma_T=0)
={}& \p_y(\abs B_t\abs \lne 1 \text{ for all } t \in [0,T]) \notag\\
={}& \frac{4}{\pi} \sum_{k=0}^\infty \frac{(-1)^k}{2k+1} \cos\left(\frac{(2k+1) \pi y}{2}\right) \e^{-\frac{(2k+1)^2\pi^2}{8}T} \notag
\\
\sim{}& \frac{4}{\pi} \cos\left(\frac{\pi y}{2}\right) \e^{-\frac{\pi^2}{8}T}, \quad  y \in (-1,1).
\end{align}
Note that, observing $$\frac{\p_y(\Gamma_T =0)}{\p_y(\Gamma_T \le s)}=\e^{-c_s T^{\frac{2}{3}} (1+o(1))} \to 0, \quad T \to \infty,$$ for a constant $c_s \gne 0$, Theorem~\ref{thm:main} may come as a surprise as~$\{\Gamma_T\leq s\}$ is a set of asymptotically strictly larger probability as compared to $\{\Gamma_T =0\}=\{\abs B_t\abs \lne 1 \text{ for all } t \in [0,T] \}$.

Let us come to the case when the Brownian motion does not start in~$[-1,1]$, as assumed until here. In this case, Corollary~\ref{cor:outside} below will imply that we cannot expect a limiting process to exist on the Wiener space~$\CC([0,\infty),\R)$ or the Skorokhod space~$\mathcal{D}([0,\infty),\R)$ since such a process would have to jump into~$(-1,1)$ immediately.
\begin{cor}
\label{cor:outside}
Fix~$\abs y\abs>1$ and set $\tau := \inf\{t \ge 0: \abs B_t\abs =1\}$. Then we have $$\lim_{T \to \infty} \p_y(\tau \ge \eps \,\abs\, \Gamma_T \le s) = 0, \quad  \text{for any $\eps \gne 0$}.$$
\end{cor}

Let us outline the general strategy behind the proof of these results. Roughly speaking, using an analysis of the Laplace transform of~$T\mapsto \p_y(\Gamma_T\leq s) \e^{\frac{\pi^2}{8} T}$ and an application of Tauberian arguments, we are able to prove Proposition~\ref{prop:asymp}. However, the required monotonicity in the Tauberian theorems seems to be difficult to establish for arbitrary starting points~$y \in (-1,1)$. We bypass this problem in two steps: First we apply the Tauberian argument to a specific initial distribution~$\nu$ supported on~$(-1,1)$ for which monotonicity of a suitable function can be established. Then we essentially show how to control the behavior of 
$$\frac{\p_y(\Gamma_T \le s)}{\p_\nu(\Gamma_T \le s) }, \quad y \in (-1,1),$$
using an elementary but insightful argument. Proposition~\ref{prop:asymp} together with tightness of the family of conditional laws will imply our main result. 

Let us now discuss some open problems and further directions of research. Similar to~\cite{BB11} and~\cite{KS16}, it is natural to ask whether one can replace the condition~$\{ \Gamma_T\leq s\}$ by conditioning more generally on~$\{ \Gamma_t\leq f(t)\,\forall t\leq T\}$ where~$f$ belongs to a suitable class of non-decreasing, non-negative functions. The question on the existence of a limiting process and its properties are expected to show an interesting dependence on the growth of~$f$.
\\
Another follow-up question consists in replacing the functional~$\Gamma_T$ by 
\begin{align*}
\tilde{\Gamma}_T:=\int_0^T\1_{\{ |B_t| \geq g(t)\}} \D t
\end{align*}
for some non-decreasing, non-negative function~$g$. The functional~$\tilde{\Gamma}_T$ measures the time a Brownian motion spends outside an interval with time-dependent boundary~$g$. Again it is clear that the growth of the function~$g$ will be the important ingredient. If the function~$g$ grows faster than $t \mapsto \sqrt{t \log(\log(t))}$, then~$\tilde{\Gamma}_{\infty}$ will be finite with probability one as a consequence of the law of the iterated logarithm and the conditioned process will essentially be a Brownian motion. Contrary, if~$g$ is strictly below the fluctuations of Brownian motion, we conjecture that the limiting process exhibits a similarly extreme entropic repulsion behavior as in the present paper, and interesting intermediate cases can be expected.
\\
Yet another open question is the precise asymptotic behavior of~$T\mapsto\p_y(\Gamma_T\leq s)$ for~$y\in \R \setminus (-1,1)$. Here we conjecture, e.g.,
\begin{align*}
\p_y(\Gamma_T \le s) \sim{}& \frac{2^{\frac{17}{6}} s^{\frac{1}{6}}}{\sqrt{3} \pi^{\frac{11}{6}} T^{\frac{2}{3}}} \exp\left( -\frac{\pi^2}{8} T + \frac{3}{2^{\frac{7}{3}}} \pi^{\frac{4}{3}} s^{\frac{1}{3}} T^{\frac{2}{3}} - \frac{3}{2^{\frac{5}{3}}} \pi^{\frac{2}{3}} s^{\frac{2}{3}} T^{\frac{1}{3}} +\frac{\pi^2+12}{24}s \right)\\
\sim{}& \left(\frac{\pi s}{2 T}\right)^{\frac{1}{3}} \p_0(\Gamma_T \le s), \quad y \in \{-1,1\},\ T \to \infty,
\end{align*}
which would, in particular, be in accordance with \eqref{eqn:outsideasymptotics}. We remark that we are able to prove the corresponding asymptotics of the Laplace transform (cf.~Lemma~\ref{l:asympShat} below) but the inversion requires a priori knowledge of monotonicity-type properties of~$T\mapsto\p_y(\Gamma_T\leq s) \e^{\frac{\pi^2}{8} T}$, which, unfortunately, we cannot ensure for \textit{any} deterministic starting point~$y \in \R$.
\\ \\
The structural outline of this paper is as follows. In Section~\ref{sec:proofofasymptotics} we prove Proposition~\ref{prop:asymp} following the strategy mentioned above: After some preliminaries in Section~\ref{sec:asympprelim}, Section~\ref{sec:laplaceanalysis} contains a thorough analysis of the Laplace transform and its asymptotic behavior. In Section~\ref{sec:inversionqsd}, this is used to determine the exact asymptotic behavior of the~$T\mapsto\p_{\nu}(\Gamma_T\leq s)$ for some very specific initial distribution~$\nu$. In Section~\ref{sec:startingpointsgeneral}, we finally study the precise dependence of~$T\mapsto\p_y(\Gamma_T\leq s)$ on the starting point~$y$. Section~\ref{sec:proofofmaintheorem} finishes the proof of Theorem~\ref{thm:main} and contains the proof of Corollary~\ref{cor:outside}.

\section{Proof of Proposition~\ref{prop:asymp}} \label{sec:proofofasymptotics}
\subsection{Preliminaries} \label{sec:asympprelim}
Let us now start with the proof of Proposition~\ref{prop:asymp}. To this end, let~$y \in [-1,1]$. Recalling the behavior of~$\p_y(\Gamma_T=0)$ cited in~\eqref{eqn:asympGamma0} for~$y \in (-1,1)$ and observing~$\p_y(\Gamma_T = 0)=0$ for~$y \in \{-1,1\}$, we may focus on the analysis of the function defined by $$R:=R_{y,s}: [0,\infty) \to [0,1], \quad T \mapsto \p_y(\Gamma_T \in (0,s]).$$ Let $$H:=\{\lambda \in \C: \op{Re}(\lambda) \gne 0\}$$ be the complex right half-plane. The key tool in deriving the asymptotics of~$R$ is Theorem~$1'$ in~\cite{Ing41}, which is of Tauberian type. It connects the asymptotic behavior of the Laplace transform of~$R$ at~$0$ with that of~$R$ at~$\infty$. However, we will not be able to apply Lemma~\ref{l:tauberianthm} directly to~$R$ due to the required regularity conditions.

\begin{lem}[Ingham '41]
\label{l:tauberianthm}
Let $S: [0,\infty) \to [0,\infty)$ be a non-decreasing function such that
\begin{equation}
\label{eqn:Shatdef}
\widehat{S}(\lambda):= \lambda \int_0^\infty S(T) \e^{-\lambda T} \D T \in \C, \quad \lambda \in H,
\end{equation}
converges. Let~$D \subseteq \C \setminus \{0\}$ be a domain containing all positive real numbers. Furthermore, let $U,V: D \to \C$ be holomorphic with $U(\lambda),V(\lambda) \in \R_{\gne 0}$ for sufficiently small~$\lambda \in \R_{\gne 0}$ and with $$\widehat{S}(\lambda) \sim U(\lambda) \e^{V(\lambda)}, \quad \text{for real } \lambda \searrow 0,$$ as well as
\begin{equation}
\label{eqn:taubOasymp}
\widehat{S}(\lambda) = O\left(U(\abs \lambda\abs)\e^{V(\abs \lambda\abs)}\right), \quad \lambda \to 0,\ \lambda \in H.
\end{equation}
Denoting by~$d: D \to [0,\infty]$ the distance function to~$D^c$, assume that there is some~$k \gne 1$ with
\begin{equation}
\label{eqn:taubcond1}
\lambda V'(\lambda) \searrow -\infty, \quad \frac{\sqrt{V''(\lambda)} }{\abs V'(\lambda)\abs} = o\left(\frac{d(\lambda)}{\lambda}\right), \quad \lambda^k V'(\lambda) \nearrow 0, \qquad \text{for real } \lambda \searrow 0,
\end{equation}
as well as 
\begin{equation}
\label{eqn:taubcond2}
\sup_{\substack{z \in \C,\, \abs z\abs \lne d(\lambda)}} \abs V''(\lambda + z)\abs = O(V''(\lambda)), \quad \sup_{\substack{z \in \C,\, \abs z\abs \lne d(\lambda)}} \abs U(\lambda + z)\abs = O(U(\lambda)), \qquad \text{for real } \lambda \searrow 0.
\end{equation}
Then we have $$S(T) \sim \frac{U(h(T))\e^{T h(T) + V(h(T))}}{h(T) \sqrt{2 \pi V''(h(T))}}, \quad T \to \infty,$$ where~$h$ is the inverse of~$-V'\abs_{(0,\eps)}$ for sufficiently small~$\eps \gne 0$.
\end{lem}

\begin{rem}
\label{rem:taub}
\begin{enumerate}[(a)]
\item In assumption~\eqref{eqn:taubcond1}, the notations~$\searrow$ and~$\nearrow$ denote monotone convergence. In particular, $\lambda V'(\lambda) \searrow -\infty$ for real~$\lambda \searrow 0$ implies $$\frac{\D}{\D \lambda} (\lambda V'(\lambda)) \ge 0 \quad \text{and} \quad V'(\lambda) \lne 0, \qquad \text{for all sufficiently small }\lambda \in \R_{\gne 0}.$$ Thus there exists an~$\eps \gne 0$ with $$V''(\lambda) \ge -\frac{V'(\lambda)}{\lambda} \gne 0, \quad \text{for all }\lambda \in (0,\eps),$$ so that~$-V'\abs_{(0,\eps)}$ is strictly decreasing. Noting~$\lim_{\lambda \searrow 0} -V'(\lambda) = \infty$, it must indeed have an inverse~$h: (T_0, \infty) \to (0,\eps)$ for a suitable~$T_0 \gne 0$.
\item In~\cite{Ing41}, the function~$S$ is assumed to satisfy~$S(0)=0$ because the result itself (but not the proof) is presented in terms of the Laplace-Stieltjes transform:
Assuming in our formulation additionally that~$S(0)=0$, an integration by parts shows that~$\widehat{S}$ is nothing but the Laplace-Stieltjes transform of~$S$ (see the beginning of the proof or, e.g., Proposition I.13.1 in~\cite{Kor04}). From the proof (as well as from the nature of the result), it is evident that the assumptions on~$S$ may be relaxed in the above way if we confine ourselves to defining~$\widehat{S}$ by~\eqref{eqn:Shatdef}.
\item In~\cite{Ing41}, assumption~\eqref{eqn:taubOasymp} is replaced by a weaker condition which does not require uniformity on all of~$H$ but only on the complex wedge $\{\lambda \in H: \abs \lambda\abs \le c \op{Re}(\lambda)\}$ for each~$c \gne 0$. The assumption that~$D$ contains the whole positive axis is relaxed in~\cite{Ing41} as well.
\end{enumerate}
\end{rem}

\subsection{Analysis of the Laplace transform}
\label{sec:laplaceanalysis}
Throughout this section, we consider~$y \in [-1,1]$. Our first step is to compute the Laplace transform of~$R$. For the purpose of readability, let us define
$$v: (0,\infty) \to \C, \quad \lambda \mapsto \sqrt{2 \lambda} \tanh(\sqrt{2 \lambda})$$ and
$$u:= u_{y}: (0,\infty) \to \C, \quad \lambda \mapsto \frac{\cosh(y \sqrt{2\lambda})}{\cosh(\sqrt{2 \lambda})}.$$
These functions are related to the functions~$U$ and~$V$ to which we will apply Ingham's Tauberian theorem (Lemma~\ref{l:tauberianthm}).
The following result essentially is formula~3.1.4.4 in~\cite{BS02}
which can be proved by Kac's theory devoloped in~\cite{Kac49} and~\cite{Kac51}:
\begin{lem}
\label{l:laplacetrafoinside}
For almost every~$T \gne 0$, we have $$R(T) = \frac{2}{\sqrt{2\pi}} \int_0^s \int_0^\infty \rho_{r,x}(T) \D x \D r,$$
where $(\rho_{r,x}: (0,\infty) \to \R)_{r,x \gne 0}$ is a family of functions with Laplace transforms given by $$(\LL(\rho_{r,x}))(\lambda) = u(\lambda) \e^{-v(\lambda)x-\lambda r- \frac{x^2}{2r}} \left(\frac{1}{\sqrt{r}}+ \frac{v(\lambda) x}{2\lambda \sqrt{r^3}}\right),
\quad \lambda,r,x \gne 0.$$
\end{lem}

The next lemma will show that one can exchange the order of integration to obtain the Laplace transform of~$R$. The resulting formula can be extended to the half-plane $$H^\leftarrow:=\left\{\lambda\in \C: \op{Re}(\lambda) \gne -\frac{\pi^2}{8}\right\}= H - \frac{\pi^2}{8}:$$

\begin{lem}
\label{l:laplacetrafooutside}
The functions~$u$ and~$v$ as well as~$\lambda \mapsto \frac{v(\lambda)}{2 \lambda}$ are well-defined and holomorphic on~$H^\leftarrow$. (The singularity of the latter in~$0$ is removable.) Moreover, the Laplace transform of~$R$ is well-defined on~$H^\leftarrow$ and satisfies
\begin{equation}
\label{eqn:laplacetrafooutside}
(\LL(R))(\lambda) = \frac{2 u(\lambda)}{\sqrt{2\pi}} \int_0^s \int_0^\infty \e^{-v(\lambda)x-\lambda r- \frac{x^2}{2r}} \left(\frac{1}{\sqrt{r}}+ \frac{v(\lambda) x}{2\lambda \sqrt{r^3}}\right) \D x \D r, \quad \lambda \in H^\leftarrow.
\end{equation}
\end{lem}
\begin{proof}
The functions~$z \mapsto z \tanh(z)$ and~$z \mapsto \frac{\tanh(z)}{z}$ as well as~$z \mapsto \frac{\cosh(yz)}{\cosh(z)}$ are well-defined (after removing the singularity of the second in~$0$), holomorphic and even on $$\left\{z \in \C: \frac{z^2}{2} \in H^\leftarrow\right\} \subseteq \C \setminus \left\{\frac{(2n+1)\pi}{2} i: n \in \Z\right\}.$$ Consequently, the three functions mentioned in the lemma are well-defined and holomorphic, independent of the chosen branch of the square root.\\
For all~$r,x \gne 0$, this implies that the Laplace transform of the function~$\rho_{r,x}$ introduced in Lemma~\ref{l:laplacetrafoinside} has a (unique) holomorphic extension to~$H^\leftarrow$. Now let~$K \subseteq H^\leftarrow$ be compact. By continuity, $\bigcup_{\lambda \in K} \{u(\lambda),v(\lambda),\frac{v(\lambda)}{2\lambda},\lambda\}$ is bounded by some~$m_K \gne 0$. Letting~$Z_r \sim \NN(m_Kr,r)$ for~$r \gne 0$, we obtain
\begin{align*}
{}&\int_0^s \int_0^\infty \sup_{\lambda \in K} \left\abs \frac{u(\lambda)}{\sqrt{2\pi}}\e^{-v(\lambda)x-\lambda r- \frac{x^2}{2r}} \cdot 1 \cdot \left(\frac{1}{\sqrt{r}}+ \frac{v(\lambda) x}{2\lambda \sqrt{r^3}}\right)\right\abs \D x\D r\\
\le{}&\int_0^s \int_0^\infty \frac{m_K}{\sqrt{2\pi}} \e^{m_K x+m_Ks-\frac{x^2}{2r}} \e^{\frac{m_K^2(s-r)}{2}} \left(\frac{1}{\sqrt{r}}+ \frac{m_K x}{\sqrt{r^3}}\right) \D x\D r\\
={}& m_K\e^{m_Ks+\frac{m_K^2s}{2}} \int_0^s \int_0^\infty \left(1+ \frac{m_Kx}{r} \right)\cdot \frac{1}{\sqrt{2 \pi r}}  \e^{-\frac{(x-m_Kr)^2}{2r}}\D x\D r\\
={}& m_K\e^{m_K s+\frac{m_K^2s}{2}} \int_0^s \E\left[\1_{\{Z_r \ge 0\}}\left(1+\frac{m_K Z_r}{r}\right)\right] \D r\\
\le{}& m_K\e^{m_K s+\frac{m_K^2s}{2}} \int_0^s 1+ \frac{m_K}{r} (m_K r + \E\abs Z_r-m_K r\abs) \D r\\
={}& m_K\e^{m_K s\frac{m_K^2s}{2}} \int_0^s 1+ m_K^2+ m_K\sqrt{\frac{2}{\pi r}} \D r\\
\lne{}& \infty.
\end{align*}
The dominated convergence theorem yields the continuity of the right hand side of~\eqref{eqn:laplacetrafooutside} as a function of~$\lambda \in H^\leftarrow$. Now let~$\gamma$ be a closed path in~$H^\leftarrow$. Bounding the path integral by the length of~$\gamma$ multiplied with the supremum of the integrand on~$K:= \op{trace}(\gamma)$, the above computation implies
$$\int_0^s \int_0^\infty \int_\gamma \left\abs u(\lambda)\e^{- v(\lambda) x-\lambda r-\frac{x^2}{2r}} \left(\frac{1}{\sqrt{r}}+ \frac{v(\lambda) x}{2\lambda \sqrt{r^3}}\right)\right\abs \D \lambda \D x\D r \lne \infty,$$
which allows us to apply Fubini's theorem. Since~$\LL(\rho_{r,x})$ is holomorphic, Cauchy's integral theorem implies
\begin{align*}
{}&\int_\gamma \int_0^s \int_0^\infty u(\lambda)\e^{-v(\lambda)x-\lambda r-\frac{x^2}{2r}} \left(\frac{1}{\sqrt{r}}+ \frac{v(\lambda) x}{2\lambda\sqrt{r^3}}\right) \D x \D r \D \lambda\\
={}& \int_0^s \int_0^\infty \int_\gamma (\LL(\rho_{r,x}))(\lambda) \D \lambda \D x \D r
=0.
\end{align*}
By Morera's theorem, $$H^\leftarrow \to \C, \quad \lambda \mapsto \frac{2u(\lambda)}{\sqrt{2\pi}}  \int_0^s \int_0^\infty \e^{-v(\lambda)x-\lambda r-\frac{x^2}{2r}} \left(\frac{1}{\sqrt{r}}+ \frac{v(\lambda) x}{2\lambda\sqrt{r^3}}\right) \D x \D r$$ is holomorphic. By Lemma~\ref{l:laplacetrafoinside} and Tonelli's theorem, it coincides with~$\LL(R)$ on~$(0,\infty)$ and hence, by the identity theorem for holomorphic function, on all of~$H^\leftarrow$.
\end{proof}

Introducing $$w: H^\leftarrow \to \C, \quad \lambda \mapsto \frac{(v(\lambda))^2}{2} - \lambda,$$ we can rewrite the Laplace transform of~$R$ in the following way, which is more suitable for our further analysis:
\begin{lem}
\label{l:laplacetrafosimplified}
For each~$\lambda \in H^\leftarrow \setminus \{0\}$, we have
\begin{align}
\label{Laplace2}
(\LL(R))(\lambda) = \frac{u(\lambda)}{\sqrt{2\pi} \lambda}\left(\sqrt{2\pi} - 2 \e^{w(\lambda) s} \int_0^\infty \e^{-\frac{(x+v(\lambda) \sqrt{s})^2}{2}} \D x + (v(\lambda)-1) \int_0^s \frac{1}{\sqrt{r}} \e^{-\lambda r} \D r\right).
\end{align}
\end{lem}
\begin{proof}
Let~$\lambda \in H^\leftarrow$. Then we have
\begin{align}
\label{eqn:laplacetrafosimplified1}
\int_0^s \int_0^\infty \e^{-v(\lambda)x\sqrt{r}-\lambda r -\frac{x^2}{2}} \frac{x+v(\lambda) \sqrt{r}}{\sqrt{r}} \D x \D r
={}& \int_0^s \e^{w(\lambda) r} \int_0^\infty \e^{-\frac{(x+v(\lambda) \sqrt{r})^2}{2}} \frac{x+v(\lambda) \sqrt{r}}{\sqrt{r}} \D x \D r \notag\\
={}& \int_0^s \e^{w(\lambda) r} \e^{-\frac{(v(\lambda) \sqrt{r})^2}{2}}\frac{1}{\sqrt{r}} \D r
= \int_0^s \frac{1}{\sqrt{r}} \e^{-\lambda r} \D r.
\end{align}
Noting
\begin{align*}
\int_0^\infty \sup_{r \in [\eps,s]} \left\abs \e^{-\frac{(x+v(\lambda) \sqrt{r})^2}{2}} \frac{x+v(\lambda) \sqrt{r}}{2 \sqrt{r}}\right\abs \D x
\le \int_0^\infty \e^{-\frac{x^2}{2} + x \abs v(\lambda)\abs \sqrt{s} +\frac{\abs v(\lambda)\abs^2 s}{2}} \frac{x+\abs v(\lambda)\abs \sqrt{s}}{2 \sqrt{\eps}} \D x \lne \infty
\end{align*}
for each~$\eps \in (0,s)$, we are allowed to differentiate w.r.t.~$r$ under the $x$-integral in the following computation. Integrating by parts w.r.t.~$r$ in the first step and using~\eqref{eqn:laplacetrafosimplified1} in the last, we get
\begin{align}
\label{eqn:laplacetrafosimplified2}
{}& \int_0^s \int_0^\infty \e^{-v(\lambda) \sqrt{r} x -\lambda r -\frac{x^2}{2}} \D x \D r\notag\\
={}& \int_0^s \e^{w(\lambda) r} \int_0^\infty \e^{-\frac{(x+v(\lambda) \sqrt{r})^2}{2}} \D x \D r\notag\\
={}& \frac{\e^{w(\lambda) s}}{w(\lambda)} \int_0^\infty \e^{-\frac{(x+v(\lambda) \sqrt{s})^2}{2}} \D x - \frac{\sqrt{2\pi}}{2w(\lambda)}
+ \int_0^s \frac{\e^{w(\lambda) r}}{w(\lambda)} \int_0^\infty \e^{-\frac{(x+v(\lambda) \sqrt{r})^2}{2}} \frac{x+v(\lambda) \sqrt{r}}{2 \sqrt{r}}\D x \D r\notag\\
={}& \frac{1}{w(\lambda)} \left(\e^{w(\lambda) s} \int_0^\infty \e^{-\frac{(x+v(\lambda) \sqrt{s})^2}{2}} \D x - \frac{\sqrt{2\pi}}{2}
+ \int_0^s \frac{1}{2\sqrt{r}} \e^{-\lambda r} \D r\right).
\end{align}
Substituting~$x$ by~$\sqrt{r}x$ in the first step and plugging in~\eqref{eqn:laplacetrafosimplified1} and~\eqref{eqn:laplacetrafosimplified2} in the third, we get
\begin{align*}
{}&\int_0^s \int_0^\infty \e^{-v(\lambda)x-\lambda r-\frac{x^2}{2r}} \left(\frac{1}{\sqrt{r}}+ \frac{v(\lambda) x}{2\lambda\sqrt{r^3}}\right) \D x \D r\\
={}& \int_0^s \int_0^\infty \e^{-v(\lambda)x\sqrt{r}-\lambda r-\frac{x^2}{2}} \left(1+ \frac{v(\lambda) x}{2\lambda\sqrt{r}}\right) \D x \D r\\
={}& \int_0^s \int_0^\infty \e^{-v(\lambda)x\sqrt{r}-\lambda r-\frac{x^2}{2}} \left(-\frac{w(\lambda)}{\lambda} + \frac{v(\lambda)}{2\lambda} \cdot \frac{x + v(\lambda) \sqrt{r}}{\sqrt{r}}\right) \D x \D r\\
={}& -\frac{w(\lambda)}{\lambda} \cdot \frac{1}{w(\lambda)} \left(\e^{w(\lambda) s} \int_0^\infty \e^{-\frac{(x+v(\lambda) \sqrt{s})^2}{2}} \D x - \frac{\sqrt{2\pi}}{2}
+ \int_0^s \frac{1}{2\sqrt{r}} \e^{-\lambda r} \D r\right) + \frac{v(\lambda)}{2\lambda} \int_0^s \frac{1}{\sqrt{r}} \e^{-\lambda r} \D r.
\end{align*}
Simplifying this term slightly and inserting it into~\eqref{eqn:laplacetrafooutside}, we obtain~\eqref{Laplace2}.
\end{proof}

Let us now start with the asymptotic analysis of~$\LL(R)$ near its rightmost singularity, i.e., near~$-\frac{\pi^2}{8}$. In view of Ingham's Tauberian theorem (Lemma~\ref{l:tauberianthm}), we translate~$-\frac{\pi^2}{8}$ to the origin. To this end, we set~$u^\rightarrow:= u\big( \,\cdot \,- \frac{\pi^2}{8}\big)$ and similarly define~$v^\rightarrow$ and~$w^\rightarrow$.
\begin{lem}
\label{l:uvwexpansion}
We have $$v^\rightarrow(\lambda) = -\frac{\pi^2}{4\lambda} + \frac{3}{2} +\frac{3+\pi^2}{3\pi^2} \lambda + O(\lambda^2), \quad \lambda \to 0,\ \lambda \in H,$$
and
\begin{align*}
w^\rightarrow(\lambda) = \frac{\pi^4}{32\lambda^2} - \frac{3\pi^2}{8\lambda} +\frac{\pi^2+21}{24} + O(\lambda), \quad \lambda \to 0,\ \lambda \in H.
\end{align*}
For~$y \in (-1,1)$, we uniformly get
\begin{align*}
u^{\rightarrow}(\lambda) \sim
\cos \left(y \sqrt{\frac{\pi^2}{4}-2 \lambda} \right)
\frac{\pi}{2\lambda}, \quad \lambda \to 0,\ \lambda \in H.
\end{align*}
\end{lem}
\begin{proof}
Recalling $$\tan(z) = -\left(z-\frac{\pi}{2}\right)^{-1} + \frac{1}{3} \left(z - \frac{\pi}{2}\right) + O\left( \left(z - \frac{\pi}{2}\right)^3 \right), \quad z \to \frac{\pi}{2},$$ we obtain $$-z \tan(z)
= \frac{\pi}{2} \left(z-\frac{\pi}{2}\right)^{-1} +1 - \frac{\pi}{6} \left(z - \frac{\pi}{2}\right) + O \left(\left(z - \frac{\pi}{2}\right)^2\right), \quad z \to \frac{\pi}{2}.$$ On the other hand, a Taylor expansion yields
\begin{equation}
\label{eqn:sqrttaylor}
\sqrt{\frac{\pi^2}{4}-2\lambda} - \frac{\pi}{2} = - \frac{2}{\pi} \lambda - \frac{4}{\pi^3}\lambda^2 - \frac{16}{\pi^5} \lambda^3 + O(\lambda^4), \quad \lambda \to 0,\ \lambda \in H.
\end{equation}
Using~$zi\tanh(zi) = -z\tan(z)$ for~$z \in \C \setminus \big\{\frac{(2n+1)\pi}{2}: n \in \Z\big\}$ in the second step and the above two expansions in the third, we get
\begin{align*}
v^{\rightarrow}(\lambda)
={}& \sqrt{2 \lambda-\frac{\pi^2}{4}} \tanh\left(\sqrt{2\lambda - \frac{\pi^2}{4}}\right)\\
={}& -\sqrt{\frac{\pi^2}{4}-2 \lambda} \tan\left(\sqrt{\frac{\pi^2}{4}-2\lambda}\right)\\
={}& \frac{\pi}{2} \left(- \frac{2}{\pi} \lambda - \frac{4}{\pi^3}\lambda^2 - \frac{16}{\pi^5} \lambda^3 + O(\lambda^4)\right)^{-1} + 1 - \frac{\pi}{6} \left(- \frac{2}{\pi} \lambda + O(\lambda^2)\right) + O(\lambda^2)\\
={}& -\frac{\pi^2}{4\lambda} \left(1 - \left(-\frac{2}{\pi^2}\lambda - \frac{8}{\pi^4} \lambda^2 + O(\lambda^3)\right)\right)^{-1} + 1 - \frac{\pi}{6} \left(- \frac{2}{\pi} \lambda + O(\lambda^2)\right) + O(\lambda^2)\\
={}& -\frac{\pi^2}{4\lambda} \left(\sum_{n=0}^\infty \left(-\frac{2}{\pi^2}\lambda - \frac{8}{\pi^4} \lambda^2 + O(\lambda^3)\right)^n\right) + 1 +\frac{\pi^2}{3\pi^2} \lambda + O(\lambda^2)\\
={}& -\frac{\pi^2}{4\lambda} \left(1-\frac{2}{\pi^2} \lambda - \frac{8}{\pi^4} \lambda^2 + \frac{4}{\pi^4} \lambda^2 + O(\lambda^3)\right) + 1 +\frac{\pi^2}{3\pi^2} \lambda + O(\lambda^2)\\
={}& -\frac{\pi^2}{4\lambda} + \frac{3}{2} + \frac{3+\pi^2}{3\pi^2} \lambda + O(\lambda^2), \quad \lambda \to 0,\ \lambda \in H.
\end{align*}
We deduce
\begin{align*}
w^\rightarrow(\lambda)
= \frac{(v^\rightarrow(\lambda))^2}{2} -\lambda + \frac{\pi^2}{8}
={}& \frac{1}{2}\left(\frac{\pi^2}{16 \lambda^2} - \frac{3\pi^2}{4 \lambda} + \frac{9}{4} - \frac{3+\pi^2}{6}\right)+ \frac{\pi^2}{8} + O(\lambda)\\
={}& \frac{\pi^4}{32\lambda^2} - \frac{3\pi^2}{8\lambda} +\frac{\pi^2+21}{24} + O(\lambda), \quad \lambda \to 0,\ \lambda \in H.
\end{align*}
Recalling~$\cosh(iz) = \cos(z)$ for each~$z \in \C$ and~$\cos(\frac{\pi}{2}-z) \sim z$ as~$z \to 0$ and using~\eqref{eqn:sqrttaylor}, we finally obtain
\begin{align*}
u^\rightarrow(\lambda) ={}& \frac{\cosh \left(y \sqrt{2 \lambda-\frac{\pi^2}{4}} \right)}{\cosh\left( \sqrt{2\lambda- \frac{\pi^2}{4}}\right)} = \frac{\cos \left(y \sqrt{\frac{\pi^2}{4}-2 \lambda} \right)}{\cos\left( \sqrt{\frac{\pi^2}{4}-2 \lambda}\right)}\\
={}& \frac{\cos \left(y \sqrt{\frac{\pi^2}{4}-2 \lambda} \right)}{\cos\left(\frac{\pi}{2}-\frac{2}{\pi}\lambda + O(\lambda^2) \right)}
\sim \frac{\cos \left(y \sqrt{\frac{\pi^2}{4}-2 \lambda} \right)}{\frac{2}{\pi}{\lambda}}, \quad \lambda \to 0,\ \lambda \in H,
\end{align*}
uniformly in~$y \in (-1,1)$.
\end{proof}

Let us now introduce the objects needed in Ingham's Tauberian theorem (Lemma~\ref{l:tauberianthm}): We define $$S:=S_y: [0,\infty) \to [0,\infty), \quad T \mapsto \e^{\frac{\pi^2}{8} T} R(T)$$ and $$\widehat{S}:=\widehat{S_y}: H \to \C, \quad \lambda \mapsto \lambda \int_0^\infty S(T) \e^{-\lambda T} \D T = \lambda \cdot  (\LL(R))\left(\lambda-\frac{\pi^2}{8}\right).$$ Moreover, we consider the complex wedge $$D:= \left\{\lambda \in H: \abs\op{arg}(\lambda)\abs \lne \frac{\pi}{4}\right\},$$ where~$\op{arg}$ is a branch of the complex argument function taking values in~$[-\pi,\pi]$, as well as the holomorphic function $$V: D \to \C, \quad \lambda \mapsto \left(\frac{\pi^4}{32\lambda^2} - \frac{3\pi^2}{8\lambda}
+ \frac{\pi^2+21}{24}\right)s.$$ We observe~$V(\lambda)
\in \R_{\gne 0}$ for sufficiently small~$\lambda \in \R_{\gne 0}$. More importantly, the asymptotic behavior of the transform~$\widehat{S}$ has the required structure:
\begin{lem}
\label{l:asympShat}
We have $$\widehat{S}(\lambda)=\widehat{S_y}(\lambda) \sim \frac{16}{\pi^2} \lambda u^\rightarrow(\lambda) \e^{V(\lambda)} \quad \text{for real } \lambda \searrow 0,$$ as well as $$\widehat{S}(\lambda)=\widehat{S_y}(\lambda) = O\left(\abs \lambda u^{\rightarrow}(\lambda) \abs \e^{V(\abs \lambda\abs)}\right), \quad \lambda \to 0,\ \lambda \in H,$$ both uniformly in~$y \in [-1,1]$.
\end{lem}
\begin{proof}
Regarding the uniformity, we note that only~$u^\rightarrow$ (implicitly) depends on~$y$. We start by observing $$\left\abs \int_0^s \frac{1}{\sqrt{r}} \e^{-\left(\lambda-\frac{\pi^2}{8}\right) r} \D r\right\abs \le \e^{\abs\lambda \abs s} \int_0^s \frac{1}{\sqrt{r}} \e^{\frac{\pi^2}{8}r} \D r = O(1), \quad \lambda \to 0,\ \lambda \in H.$$
Together with Lemmas~\ref{l:laplacetrafosimplified} and~\ref{l:uvwexpansion}, we get
\begin{align}
\label{eqn:asympShat}
\widehat{S}(\lambda) ={}& \frac{\lambda u^{\rightarrow}(\lambda)}{\sqrt{2\pi}\left(\lambda - \frac{\pi^2}{8}\right)}\left(\sqrt{2\pi} - 2 \e^{w^{\rightarrow}(\lambda) s} \int_0^\infty \e^{-\frac{(x+v^{\rightarrow}(\lambda) \sqrt{s})^2}{2}} \D x + (v^{\rightarrow}(\lambda)-1) \int_0^s \frac{1}{\sqrt{r}} \e^{-\left(\lambda-\frac{\pi^2}{8}\right) r} \D r\right)\notag\\
={}& \frac{16 \lambda u^{\rightarrow}(\lambda)}{\pi^2\sqrt{2\pi}}(1+o(1)) \left(\e^{w^{\rightarrow}(\lambda) s} \int_0^\infty \e^{-\frac{(x+v^{\rightarrow}(\lambda) \sqrt{s})^2}{2}} \D x + O\left(\frac{1}{\lambda}\right)\right), \quad \lambda \to 0,\ \lambda \in H,
\end{align}
uniformly in~$y \in [-1,1]$. Since Lemma~\ref{l:uvwexpansion} implies $\lim_{\lambda \searrow 0} v^{\rightarrow}(\lambda)\sqrt{s} = -\infty$, we can deduce
\begin{align*}
\widehat{S}(\lambda)
={}& \frac{16}{\pi^2} \lambda u^\rightarrow(\lambda) (1+o(1)) \left(\e^{w^{\rightarrow}(\lambda) s} \int_{v^{\rightarrow}(\lambda)\sqrt{s}}^\infty \frac{1}{\sqrt{2\pi}}\e^{-\frac{x^2}{2}} \D x + O\left(\frac{1}{\lambda}\right)\right)\\
\sim{}& \frac{16}{\pi^2} \lambda u^\rightarrow(\lambda) \e^{V(\lambda)}, \quad \text{for real } \lambda \searrow 0,
\end{align*}
uniformly in~$y \in [-1,1]$. Since Lemma~\ref{l:uvwexpansion} also implies
\begin{align*}
\frac{(\op{Re} v^\rightarrow(\lambda))^2}{2}
={}& \frac{1}{2} \left(- \frac{\pi^2}{4} \op{Re}\frac{1}{\lambda} + \frac{3}{2} + O(\abs \lambda\abs) \right)^2
= \frac{\pi^4}{32}\left(\op{Re}\frac{1}{\lambda}\right)^2 - \frac{3 \pi^2}{8} \op{Re}\frac{1}{\lambda} + O(1)\\
\le {}& \frac{\pi^4}{32}\left\abs \frac{1}{\lambda}\right\abs^2 - \frac{3 \pi^2}{8} \left\abs\frac{1}{\lambda}\right\abs + O(1)
= \frac{V(\abs \lambda\abs)}{s} + O(1), \quad \lambda \to 0,\ \lambda \in H,
\end{align*}
we obtain
\begin{align*}
\left\abs \e^{w^{\rightarrow}(\lambda) s} \int_0^\infty \e^{-\frac{(x+v^{\rightarrow}(\lambda) \sqrt{s})^2}{2}} \D x \right\abs
\le{}& \int_0^\infty \left\abs \e^{-\frac{x^2}{2} - x v^{\rightarrow}(\lambda) \sqrt{s} - \left(\lambda-\frac{\pi^2}{8}\right) s}\right\abs \D x \\
={}& \int_0^\infty \e^{-\frac{x^2}{2} - x \op{Re} v^{\rightarrow}(\lambda) \sqrt{s} - \op{Re} \lambda s + \frac{\pi^2}{8} s} \D x \\
={}& \e^{\frac{(\op{Re} v^{\rightarrow}(\lambda))^2}{2}s} \e^{- \op{Re} \lambda s + \frac{\pi^2}{8} s} \int_0^\infty \e^{-\frac{(x+\op{Re}v^{\rightarrow}(\lambda) \sqrt{s})^2}{2}} \D x\\
={}& O\big(\e^{ V(\abs\lambda\abs)}\big) O(1), \quad \lambda \to 0,\ \lambda \in H.
\end{align*}
Plugging this into~\eqref{eqn:asympShat}, we deduce
\begin{align*}
\big\abs \widehat{S}(\lambda)\big\abs
={}& \frac{16 \abs \lambda u^{\rightarrow}(\lambda) \abs}{\pi^2\sqrt{2\pi}} (1+o(1)) \left\abs \e^{w^{\rightarrow}(\lambda) s} \int_0^\infty \e^{-\frac{(x+v^{\rightarrow}(\lambda) \sqrt{s})^2}{2}} \D x + O\left(\frac{1}{\lambda}\right) \right\abs\\
={}& O\big(\abs \lambda u^{\rightarrow}(\lambda) \abs\e^{V(\abs \lambda\abs)}\big), \quad \lambda \to 0,\ \lambda \in H,
\end{align*}
uniformly in~$y \in [-1,1]$.
\end{proof}

\subsection{Inversion when starting in the qsd} \label{sec:inversionqsd}
As already noted in Section~\ref{sec:asympprelim}, it seems unclear how to prove monotonicity of~$S_y$, which is required to apply Ingham's Tauberian theorem (Lemma~\ref{l:tauberianthm}), for any deterministic starting point~$y \in \R$. For a specific starting distribution~$\nu$, however, the monotonicity of the corresponding function is rather easy to check. This allows us to rigorously derive the asymptotic behavior of~$\p_\nu(\Gamma_T \in (0,s])$ as~$T \to \infty$ in a first step. Let $$\tau:= \inf\{t \ge 0: \abs B_t\abs = 1\}$$ be the first exit time of~$B$ from the interval~$(-1,1)$. We recall that there is a unique distribution~$\nu$ supported on~$(-1,1)$ which satisfies
\begin{align}
\label{eqn:qsd}
\mathbb{P}_{\nu}(\tau \geq t, B_{t}\in \D y)= \e^{-\frac{\pi^2}{8} t} \D\nu(y), \quad y \in (-1,1),\ t\geq 0.
\end{align}
This distribution is called quasi-stationary distribution~(qsd) in the literature. It has a Lebesgue density given by
\begin{align}
\label{eqn:qsddensity}
\D\nu(y)= \frac{\pi}{4} \cos\left(\frac{\pi y}{2}\right) \D y, \quad y \in (-1,1).
\end{align}
Note that the dependence on~$y$ is the same as in Proposition~\ref{prop:asymp}. This connection will become clear in Section~\ref{sec:startingpointsgeneral}. For more information and results concerning quasi-stationary distributions, we refer to~\cite{SE07}, \cite{KS12} and~\cite{CV16}. 
The monotonicity of $$S_\nu:[0,\infty) \to  [0,\infty), \quad T\mapsto \e^{\frac{\pi^2}{8}T}\mathbb{P}_{\nu}(\Gamma_T \in (0,s])$$ now follows straight from the Markov property and~\eqref{eqn:qsd}: For all~$T,t\ge 0$, we get 
\begin{align*}
\mathbb{P}_{\nu}(\Gamma_{T+t} \in (0,s])
\geq \mathbb{P}_{\nu}(\Gamma_{T+t}\in (0,s],\tau\ge t)
= \mathbb{E}_{\nu}\left(\1_{\{ \tau\ge t\}}\mathbb{P}_{B_t}(\Gamma_T\in (0,s])\right)
= \e^{-\frac{\pi^2}{8} t}\mathbb{P}_{\nu}(\Gamma_T \in (0,s]).
\end{align*}
The function~$V$ introduced before Lemma~\ref{l:asympShat} also satisfies the technical assumptions of Ingham's Tauberian theorem (Lemma~\ref{l:tauberianthm}):
\begin{lem}
\label{l:Vcond}
The function~$V$ satisfies the conditions stated in~\eqref{eqn:taubcond1} and~\eqref{eqn:taubcond2}.
\end{lem}
\begin{proof}
The derivatives of~$V$ are given by $$V'(\lambda) = \left(-\frac{\pi^4}{16 \lambda^3} + \frac{3\pi^2}{8\lambda^2} \right)s  \quad \text{and} \quad V''(\lambda) = \left(\frac{3\pi^4}{16 \lambda^4} - \frac{3 \pi^2}{4\lambda^3} \right)s, \qquad \lambda \in D.$$ In particular, we have~$\lambda V'(\lambda) \searrow -\infty$ and~$\lambda^4 V'(\lambda) \nearrow 0$ for real~$\lambda \searrow 0$. Noting that the distance of~$\lambda$ to~$D^c$ is given by~$d(\lambda) = \frac{\lambda}{\sqrt{2}}$ for~$\lambda \in D \cap \R_{\gne 0}$, we get $$\frac{\sqrt{V''(\lambda)}}{\abs V'(\lambda)\abs} = o(1) = o\left(\frac{d(\lambda)}{\lambda}\right), \quad \text{for real } \lambda \searrow 0.$$ For every sufficiently small~$\lambda \in \R_{\gne 0}$ and each~$z \in \C$ with~$\abs z\abs \lne d(\lambda)$, we obtain 
\begin{align*}
\frac{\abs V''(\lambda + z)\abs }{V''(\lambda)} = \frac{\lambda^4}{\abs\lambda+z\abs^4} \cdot \left\abs\frac{\frac{3\pi^4}{16} - \frac{3}{4} \pi^2 (\lambda+z)}{\frac{3\pi^4}{16} - \frac{3}{4} \pi^2 \lambda}\right\abs \le \left(\frac{1}{1-\frac{1}{\sqrt{2}}}\right)^4\cdot 2,
\end{align*}
verifying the conditions stated in~\eqref{eqn:taubcond2}.
\end{proof}

As seen in Remark~\ref{rem:taub}(a), there are~$\eps,T_0 \gne 0$ such that~$-V'\abs_{(0,\eps)}$ has an inverse~$h: (T_0,\infty) \to (0,\eps)$. This function can be determined explicitly and behaves as follows:
\begin{lem}
\label{l:asymph}
The function~$h$ satisfies $$h(T) \sim \left(\frac{\pi^4 s}{2^4 T}\right)^{\frac{1}{3}} \quad \text{and} \quad h(T) \sqrt{2 \pi V''(h(T))} \sim \frac{\sqrt{3} \pi^{\frac{7}{6}}}{2^{\frac{1}{6}}} s^{\frac{1}{6}} T^{\frac{1}{3}}, \qquad T \to \infty,$$ as well as $$T h(T) + V(h(T)) = \frac{3}{2^{\frac{7}{3}}}\pi^{\frac{4}{3}} s^{\frac{1}{3}} T^{\frac{2}{3}} - \frac{3}{2^{\frac{5}{3}}} \pi^{\frac{2}{3}} s^{\frac{2}{3}} T^{\frac{1}{3}} + \frac{\pi^2+12}{24}s + O\left(\frac{1}{T^{\frac{1}{3}}}\right), \qquad T \to \infty.$$
\end{lem}
\begin{proof}
Recalling $$V'(\lambda) = \left(-\frac{\pi^4}{2^4 \lambda^3} + \frac{3\pi^2}{2^3\lambda^2} \right)s, \quad \lambda \gne 0,$$ the function~$h$ must satisfy $$\frac{T}{s} (h(T))^3 = \frac{\pi^4}{2^4} - \frac{3 \pi^2}{2^3}  h(T), \quad T \in (T_0,\infty),$$ so that Cardano's formula yields
\begin{align*}
h(T)
={}& \left(\frac{\pi^4s}{2^5 T} + \sqrt{\frac{\pi^8 s^2}{2^{10} T^2}+ \frac{\pi^6 s^3}{2^9 T^3}}\right)^{\frac{1}{3}} + \left(\frac{\pi^4s}{2^{10} T} - \sqrt{\frac{\pi^8s^2}{2^{10} T^2}+ \frac{\pi^6 s^3}{2^9 T^3}}\right)^{\frac{1}{3}}\\
={}& \left(\frac{\pi^4s}{2^5 T}\right)^{\frac{1}{3}} \left(\left(\sqrt{1+ \frac{2 s}{\pi^2 T}} +1\right)^{\frac{1}{3}} - \left(\sqrt{1+ \frac{2 s}{\pi^2 T}}-1\right)^{\frac{1}{3}}\right), \quad T \in (T_0,\infty).
\end{align*}
This implies~$h(T) \sim \big(\frac{\pi^4 s}{2^4 T}\big)^{\frac{1}{3}}$ as~$T \to \infty$. We deduce $$h(T) \sqrt{2 \pi V''(h(T))} = \sqrt{2 \pi \left(\frac{3\pi^4}{2^4 (h(T))^2} - \frac{3 \pi^2}{4 h(T)}\right)s} \sim \sqrt{\frac{3\pi^5s}{2^3} \left(\frac{2^4 T}{\pi^4 s}\right)^{\frac{2}{3}}} = \frac{\sqrt{3} \pi^{\frac{7}{6}}}{2^{\frac{1}{6}}} s^{\frac{1}{6}} T^{\frac{1}{3}}, \quad T \to \infty.$$
Observing
\begin{align}
\label{eqn:asymphsqrt}
\left(\sqrt{1+z}-1\right)^{\frac{1}{3}}= z^{\frac{1}{3}}\cdot \frac{1}{\big(\sqrt{1+z}+1\big)^{\frac{1}{3}}} = z^{\frac{1}{3}} \left(\frac{1}{2^{\frac{1}{3}}} +O(z)\right), \quad z \to 0,
\end{align}
we obtain
\begin{align*}
T h(T)
={}& \left(\frac{\pi^4s T^2}{2^5}\right)^{\frac{1}{3}} \left(\left(\sqrt{1+ \frac{2 s}{\pi^2 T}} +1\right)^{\frac{1}{3}} - \left(\sqrt{1+ \frac{2 s}{\pi^2 T}}-1\right)^{\frac{1}{3}}\right)\\
={}& \left(\frac{\pi^4s T^2}{2^5}\right)^{\frac{1}{3}} \left(\left(2^{\frac{1}{3}}+O\left(\frac{1}{T}\right)\right) - \left(\frac{2 s}{\pi^2 T}\right)^{\frac{1}{3}} \left(\frac{1}{2^{\frac{1}{3}}} +O\left(\frac{1}{T}\right)\right)\right)\\
={}& \frac{\pi^{\frac{4}{3}}}{2^{\frac{4}{3}}} s^{\frac{1}{3}} T^{\frac{2}{3}} - \frac{\pi^{\frac{2}{3}}}{2^{\frac{5}{3}}} s^{\frac{2}{3}} T^{\frac{1}{3}} + O\left(\frac{1}{T^{\frac{1}{3}}}\right), \quad T \to \infty.
\end{align*}
Similarly, \eqref{eqn:asymphsqrt} yields
\begin{align*}
H(z):={}&\frac{1}{\big(\sqrt{1+z}+1\big)^{\frac{1}{3}}-\big(\sqrt{1+z}-1\big)^{\frac{1}{3}}}
= \frac{1}{2}\left(\big(\sqrt{1+z}+1\big)^{\frac{2}{3}}+z^{\frac{1}{3}}+\big(\sqrt{1+z}-1\big)^{\frac{2}{3}}\right)\\
={}& \frac{1}{2}\left(\left(2^{\frac{2}{3}}+O(z)\right)+ z^{\frac{1}{3}}+z^{\frac{2}{3}} \left(\frac{1}{2^{\frac{2}{3}}} + O(z) \right)\right)
= \frac{1}{2^{\frac{1}{3}}} + \frac{1}{2} z^{\frac{1}{3}} + \frac{1}{2^{\frac{5}{3}}} z^{\frac{2}{3}} + O(z), \quad z \to 0,
\end{align*}
which implies
\begin{align*}
V(h(T))
={}& \frac{\pi^4 s}{2^5} \left(\frac{2^5 T}{\pi^4 s}\right)^{\frac{2}{3}} \left(H\left(\frac{2 s}{\pi^2 T}\right)\right)^2 - \frac{3 \pi^2 s}{8} \left(\frac{2^5 T}{\pi^4 s}\right)^{\frac{1}{3}} H\left(\frac{2 s}{\pi^2 T}\right) + \frac{\pi^2+21}{24}s\\
={}& \frac{\pi^4 s}{2^5} \left(\frac{2^5 T}{\pi^4 s}\right)^{\frac{2}{3}} \left(\frac{1}{2^{\frac{2}{3}}} + \frac{1}{2^{\frac{1}{3}}} \left(\frac{2 s}{\pi^2 T}\right)^{\frac{1}{3}} + \left(\frac{1}{2^2}+\frac{1}{2}\right) \left(\frac{2 s}{\pi^2 T}\right)^{\frac{2}{3}} + O\left(\frac{1}{T}\right)\right)\\
{}&\quad - \frac{3 \pi^2 s}{2^3} \left(\frac{2^5 T}{\pi^4 s}\right)^{\frac{1}{3}} \left(\frac{1}{2^{\frac{1}{3}}} + \frac{1}{2} \left(\frac{2 s}{\pi^2 T}\right)^{\frac{1}{3}} + O\left(\frac{1}{T^{\frac{2}{3}}}\right)\right) + \frac{\pi^2+21}{24}s\\
={}& \frac{\pi^{\frac{4}{3}}}{2^{\frac{7}{3}}}s^{\frac{1}{3}} T^{\frac{2}{3}} - \frac{2 \pi^{\frac{2}{3}}}{2^{\frac{5}{3}}} s^{\frac{2}{3}} T^{\frac{1}{3}} + \frac{\pi^2+12}{24}s + O\left(\frac{1}{T^{\frac{1}{3}}}\right) , \quad T \to \infty.
\end{align*}
The claim follows by adding the expansions of~$Th(T)$ and~$V(h(T))$.
\end{proof}

Integrating in Lemma~\ref{l:asympShat} w.r.t.~$\nu$, we can finally apply Ingham's Tauberian theorem (Lemma~\ref{l:tauberianthm}) to obtain the asymptotic behavior of~$S_\nu$:
\begin{prop}
\label{prop:asympSnu}
We have $$S_\nu(T) \sim \frac{ 2^{\frac{7}{6}}}{\sqrt{3} \pi^{\frac{7}{6}} s^{\frac{1}{6}} T^{\frac{1}{3}}} \exp\left( \frac{3}{2^{\frac{7}{3}}} \pi^{\frac{4}{3}} s^{\frac{1}{3}} T^{\frac{2}{3}} - \frac{3}{2^{\frac{5}{3}}} \pi^{\frac{2}{3}} s^{\frac{2}{3}} T^{\frac{1}{3}} +\frac{\pi^2+12}{24}s \right), \quad T \to \infty.$$
\end{prop}
\begin{proof}
Equation~\eqref{eqn:qsddensity} implies $\int_{-1}^1 \frac{8}{\pi} \cos\left(\frac{\pi y}{2}\right) \D \nu(y) = 2$. The constant function $U: D \to (0,\infty)$,~$\lambda \mapsto 2$ trivially satisfies all requirements of Lemma~\ref{l:tauberianthm}.
We define $$\widehat{S_\nu}:H \to \C, \quad \lambda \mapsto \lambda \int_0^\infty S_\nu(T) \e^{-\lambda T} \D T.$$
Fubini's theorem yields $$\widehat{S_\nu}(\lambda) = \lambda \int_0^\infty \int_{-1}^1 S_y(T) \e^{-\lambda T} \D \nu(y) \D T = \lambda \int_{-1}^1 \int_0^\infty S_y(T) \e^{-\lambda T} \D T \D \nu(y) =  \int_{-1}^1 \widehat{S_y}(\lambda) \D \nu(y), \quad \lambda \in H.$$ 
Using the uniformity in Lemma~\ref{l:asympShat} and in the last assertion of Lemma~\ref{l:uvwexpansion} in the first step as well as the dominated convergence theorem in the second, we obtain
\begin{align*}
\int_{-1}^1 \widehat{S_y}(\lambda) \D \nu(y)
\sim{}& \int_{-1}^1 \frac{8}{\pi} \cos \left(y \sqrt{\frac{\pi^2}{4}-2 \lambda} \right) \e^{V(\lambda)} \D\nu(y)\\
\to{}& \int_{-1}^1 \frac{8}{\pi} \cos \left(\frac{\pi y}{2} \right) \e^{V(\lambda)} \D\nu(y)
= U(\lambda) \e^{V(\lambda)}, \quad \text{for real } \lambda \searrow 0.
\end{align*}
Setting~$\C_1:= \{\lambda \in \C: \abs \lambda \abs \le 1\}$, continuity implies $$\sup_{\lambda \in \C_1, y \in [-1,1]} \left\abs \cos \left( y \sqrt{\frac{\pi^2}{4}-2 \lambda}\right)\right\abs \lne \infty.$$ Again using the uniformity in Lemma~\ref{l:asympShat} and in the last assertion of Lemma~\ref{l:uvwexpansion} in the first step, we obtain
\begin{align*}
\int_{-1}^1 \widehat{S_y}(\lambda) \D \nu(y)
=& O\left(\int_{-1}^1 \frac{8}{\pi} \left\abs \cos \left( y \sqrt{\frac{\pi^2}{4}-2 \lambda} \right) \right\abs \e^{V(\abs \lambda\abs)} \D \nu(y)\right)\\
=& O\left(U(\abs \lambda\abs) \e^{V(\abs\lambda\abs)}\right), \quad \lambda \to 0,\ \lambda \in H.
\end{align*}
Recalling Lemma~\ref{l:Vcond}, Lemma~\ref{l:tauberianthm} is applicable and implies $$\widehat{S_\nu}(\lambda) \sim \frac{U(h(T)) \e^{Th(T)+V(h(T))}}{h(T) \sqrt{2\pi V''(h(T))}}, \quad T \to \infty.$$ The claim follows by inserting the asymptotics developed in Lemma~\ref{l:asymph}.
\end{proof}

\subsection{Asymptotics for deterministic starting points} \label{sec:startingpointsgeneral}
In this section, we finally prove Proposition~\ref{prop:asymp}, i.e., we extend the precise asymptotics to the case of deterministic starting points. To this end, we discuss how~$\frac{\p_y(\Gamma_T \in (0,s])}{\p_\nu(\Gamma_T  \in (0,s])}$ behaves as~$T \to \infty$. Apart from the analysis in Section~\ref{sec:laplaceanalysis}, our argument relies on a rather elementary observation, which we have not seen used before.
Moreover, the proof will explicitly highlight the role of the density of~$\nu$ as noted in the introduction of the previous section.\\
We start with an auxiliary result, which, again, is a consequence of our analysis in Section~\ref{sec:laplaceanalysis} and Ingham's Tauberian theorem (Lemma~\ref{l:tauberianthm}):
\begin{lem}
\label{l:asympintegralS1}
For any~$T_0>0$, we have $$\lim_{T \to \infty} \frac{\int_{T-T_0}^T S_1(t) \D t}{\int_0^T S_1(t) \D t}=0.$$
\end{lem}
\begin{proof}
Clearly, the constant function $U: D \to \C$,~$\lambda \mapsto \frac{16}{\pi^2}$ satisfies the assumptions of Lemma~\ref{l:tauberianthm}. Integrating by parts, noting~$u_1(\lambda)=1$ for all $\lambda$ and applying Lemma~\ref{l:asympShat}, we get $$\lambda \int_0^\infty \left(\int_0^T S_1(t) \D t\right) \e^{-\lambda T} \D \lambda = \frac{1}{\lambda}\widehat{S_1}(\lambda) \sim U(\lambda) \e^{V(\lambda)}, \quad \text{for real } \lambda \searrow 0,$$ and similarly $$\lambda \int_0^\infty \left(\int_0^T S_1(t) \D t\right) \e^{-\lambda T} \D \lambda
= O\left( U(\abs\lambda\abs) \e^{V(\abs\lambda\abs)}\right), \quad \lambda \to 0,\ \lambda \in H.$$ Since~$S_1$ is non-negative, its primitive function is non-decreasing. Recalling Lemma~\ref{l:Vcond}, Lemma~\ref{l:tauberianthm} is applicable and implies $$\int_0^T S_1(t) \D t \sim \frac{ U(h(T)) \e^{Th(T)+V(h(T))}}{h(T) \sqrt{2\pi V''(h(T))}}, \quad T \to \infty.$$ Inserting the asymptotics developed in Lemma~\ref{l:asymph}, we deduce $$\int_0^T S_1(t) \D t \sim \frac{2^{\frac{25}{6}}}{\sqrt{3} \pi^{\frac{19}{6}} s^{\frac{1}{6}} T^{\frac{1}{3}}} \exp\left( \frac{3}{2^{\frac{7}{3}}} \pi^{\frac{4}{3}} s^{\frac{1}{3}} T^{\frac{2}{3}} - \frac{3}{2^{\frac{5}{3}}} \pi^{\frac{2}{3}} s^{\frac{2}{3}} T^{\frac{1}{3}} +\frac{\pi^2+12}{24}s \right), \quad T \to \infty.$$ Recalling $\lim_{T \to \infty} (T^q - (T-T_0)^q) = 0$ for all~$q \in (0,1)$, we deduce $$\lim_{T \to \infty} \frac{\int_0^{T-T_0} S_1(t) \D t}{\int_0^T S_1(t) \D t}=1, \quad T_0 \gne 0,$$ proving the claim.
\end{proof}

We remark that the argument in the above proof to obtain the asymptotics of~$\int_0^T S_y(t) \D t$ works for any starting point~$y \in \R$.
\\
Let us define
\begin{align*}
\sigma:=\inf\left\{T>0 : \int_0^T \1_{\{\abs X_t\abs \ge 1\}} \D t >s\right\},
\end{align*}
where~$(X_t)_{t \ge 0}$ denotes a Brownian motion with start in~$1$, independent of~$B$. In the sequel, we shall split a path with $\tau\leq T$ into a piece of length $\tau$ and an ingredient involving $\sigma$.

\begin{lem}
\label{l:expcomparison}
Given~$\lambda_1 \gne \lambda_0:=\frac{\pi^2}{8} \gne 0$, let~$e_0 \sim \op{Exp}(\lambda_0)$ and~$e_1 \sim \op{Exp}(\lambda_1)$ be independent of~$\sigma$. Then it holds
\begin{align*}
\lim_{T \to \infty} \frac{\mathbb{P}(e_1 + \sigma >T, e_1 \le T)}{\mathbb{P}(e_0 + \sigma >T, e_0 \le T)}
= \lim_{T \to \infty} \frac{\int_0^T \lambda_1 \e^{-\lambda_1 t}\mathbb{P}(\sigma >T-t) \D t}{\int_0^T \lambda_0 \e^{-\lambda_0 t}\mathbb{P}(\sigma >T-t)\D t}
=0.
\end{align*}
\end{lem}
\begin{proof}
Let~$T \ge 0$. We start by observing
\begin{align*}
\mathbb{P}(e_0 + \sigma >T, e_0 \le T){}&=\int_0^T\lambda_0 \e^{-\lambda_0 t}\mathbb{P}(\sigma>T-t) \D t\\
{}&= \lambda_0 \e^{-\lambda_0 T}\int_0^T S_1(T-t) \D t
= \lambda_0 \e^{-\lambda_0 T}\int_0^T S_1(t) \D t.
\end{align*}
In the same way, we obtain 
\begin{align*}
\mathbb{P}(e_1 + \sigma >T, e_1 \le T){}&=\int_0^T \lambda_1 \e^{-\lambda_1 t}\mathbb{P}(\sigma>T-t) \D t
=\lambda_1 \e^{-\lambda_0 T}\int_0^T \e^{-(\lambda_1-\lambda_0)t}S_1(T-t) \D t.
\end{align*}
Now let~$\eps \gne 0$. Taking~$T_{\varepsilon}>0$ such that~$\e^{-(\lambda_1-\lambda_0)t}\leq \varepsilon$ for all~$t\geq T_{\varepsilon}$, we estimate
\begin{align*}
\int_0^T\e^{-(\lambda_1-\lambda_0)t}S_1(T-t) \D t
={}& \int_0^{T_{\varepsilon}} \e^{-(\lambda_1-\lambda_0)t}S_1(T-t) \D t +\int_{T_{\varepsilon}}^T \e^{-(\lambda_1-\lambda_0)t}S_1(T-t) \D t \\
\leq{}& \int_0^{T_{\varepsilon}} S_1(T-t) \D t+\varepsilon \int_{T_{\varepsilon}}^T S_1(T-t) \D t
= \int_{T-T_{\varepsilon}}^T S_1(t) \D t+  \varepsilon \int_{0}^{T-T_{\varepsilon}} S_1(t) \D t.
\end{align*}
We conclude
\begin{align*}
\frac{\mathbb{P}(e_1 + \sigma >T, e_1 \le T)}{\mathbb{P}(e_0 + \sigma >T, e_0 \le T)} {}&\leq  \frac{\lambda_1 \left(\int_{T-T_{\varepsilon}}^T S_1(t) \D t+  \varepsilon \int_{0}^{T-T_{\varepsilon}} S_1(t) \D t \right)}{\lambda_0 \int_0^T S_1(t) \D t}
\leq \frac{\lambda_1 \int_{T-T_{\varepsilon}}^T S_1(t) \D t}{\lambda_0 \int_{0}^T S_1(t) \D t}+\frac{\lambda_1}{\lambda_0} \varepsilon, \quad T \ge T_\eps.
\end{align*}
First letting~$T \to \infty$, invoking Lemma~\ref{l:asympintegralS1}, and then~$\eps \searrow 0$, we get the claim.
\end{proof}

We are now ready to control the behavior of~$\frac{\mathbb{P}_y(\Gamma_T\in (0,s])}{\mathbb{P}_\nu(\Gamma_T\in (0,s])}$ as announced at the beginning of this section:
\begin{lem}
\label{l:quotientPyPnu}
Let~$y \in (-1,1)$. It holds that
\begin{align*}
\lim_{T \to \infty} \frac{\mathbb{P}_y(\Gamma_T\in (0,s])}{\mathbb{P}_\nu(\Gamma_T\in (0,s])} = \frac{4}{\pi} \cos\left(\frac{\pi y}{2}\right).
\end{align*}
\end{lem}
\begin{proof}
Defining $$\lambda_n:= \frac{(2n+1)^2\pi^2}{8} \quad \text{and} \quad \varphi_n(y) := \frac{4(-1)^n}{\pi(2n+1)} \cos\left(\frac{(2n+1) \pi y}{2}\right), \quad y \in (-1,1),\ n \in \N_0,$$
a Lebesgue density~$f_y: (0,\infty) \to (0,\infty)$ of~$\tau$ under~$\p_y$ is given by (see, e.g., Example~5a in~\cite{DS53}, cf.~equation~\eqref{eqn:asympGamma0})
\begin{equation}\label{e:exit2}
f_y(t)=\sum_{n=0}^{\infty} \lambda_n \e^{-\lambda_n t}\varphi_n(y), \quad t \gne 0.
\end{equation}
The dominated convergence theorem implies that~$f_y$ is continuous. Now let~$\tau_1$ and~$\tau_{-1}$ be the first times~$B$ hits~$1$ and~$-1$, respectively. Noting $\tau=\tau_{-1}\wedge\tau_1$, we observe $$\p_y(\tau \in A) \le \p_y(\tau_1 \in A) + \p_y(\tau_{-1} \in A), \quad A \in \BB((0,\infty)).$$ Recalling that~$\tau_1$ and~$\tau_{-1}$ follow inverse-chi-squared distributions (see, e.g., Remark~2.8.3 in~\cite{KS91}) and thus have bounded Lebesgue densities under~$\p_y$, the continuous density~$f_y$ of~$\tau$ must be bounded as well. Therefore, 
\begin{equation*}
g_y: (0,\infty) \to \R, \quad  t\mapsto f_y(t)-\varphi_0(y) \lambda_0 \e^{-\lambda_0 t} 
\end{equation*}
is a bounded continuous function. Furthermore, equation~\eqref{e:exit2} implies
\begin{align*}
\abs g_y(t) \abs =\e^{-\lambda_1t} \left|\sum_{k=1}^\infty \varphi_k(y) \lambda_k \e^{-(\lambda_k-\lambda_1) t} \right|
\leq \e^{-\lambda_1t}\sum_{k=1}^\infty |\varphi_k(y)| \lambda_k \e^{-(\lambda_k-\lambda_1)}, \quad t \ge 1.
\end{align*}
Noting that the last series converges and recalling the boundedness of~$g_y$, there exists a constant~$K>0$ with~$\abs g_y(t)\abs \le  K \e^{-\lambda_1 t}$ for all~$t\gne 0$. Consequently, Lemma~\ref{l:expcomparison} yields
\begin{align*}
\frac{\big\abs \int_0^T g_y(t)\mathbb{P}(\sigma >T-t) \D t \big\abs }{\int_0^T \lambda_0 \e^{-\lambda_0 t}\mathbb{P}(\sigma >T-t)\D t}
\le 
\frac{K\int_0^T \e^{-\lambda_1 t}\mathbb{P}(\sigma >T-t) \D t}{\lambda_0 \int_0^T \e^{-\lambda_0 t}\mathbb{P}(\sigma >T-t)\D t} \to 0, \quad T \to \infty.
\end{align*}
We deduce
\begin{align*}
\mathbb{P}_y(\Gamma_T\in (0,s])
={}&\mathbb{P}_y( \tau+\sigma >T, \tau \le T)\\
={}& \int_0^T f_y(t)\mathbb{P}(\sigma >T-t) \D t \\
={}& \varphi_0(y) \int_0^T\lambda_0 \e^{-\lambda_0 t} \mathbb{P}(\sigma >T-t)\D t  + \int_0^T g_y(t)\mathbb{P}(\sigma >T-t) \D t\\
\sim {}& \varphi_0(y) \int_0^T\lambda_0 \e^{-\lambda_0 t} \mathbb{P}(\sigma >T-t)\D t, \quad T \to \infty.
\end{align*}
On the other hand, the characterization of the quasi-stationary distribution~$\nu$ given in~\eqref{eqn:qsd} implies
\begin{align*}
    \mathbb{P}_{\nu}(\Gamma_T \in (0,s])= \mathbb{P}_{\nu}( \tau+\sigma >T, \tau \le T) = 
    \int_0^T\lambda_0e^{-\lambda_0 t}\mathbb{P}(\sigma>T-t) \D t, \quad T \ge 0,
\end{align*}
proving the claim.
\end{proof}

For the sake of completeness, we record the following result, which is clear from a heuristic point of view. The proof is an easy coupling argument.
\begin{lem}\label{l:comparison}
For all~$x,y \in \R$ with~$\abs y\abs \geq \abs x\abs$, we have 
\begin{equation*}
\mathbb{P}_{y}\left(\Gamma_T\leq s\right)\leq \mathbb{P}_{x}\left(\Gamma_T\leq s\right).
\end{equation*}
\end{lem}
\begin{proof}
W.l.o.g., assume~$x,y \ge 0$. Let~$(B^1_t)_{t\geq 0}$ and~$(B^2_t)_{t\geq 0}$ be independent standard Brownian motions. Using the notation $\tau_z=\inf\{ t\geq 0 : B^1_t=z\}$ for~$z \in \R$, the processes~$(X_t^x)_{t \geq 0}$ and~$(X_t^y)_{t \geq 0}$ defined by
\begin{equation*}
X_t^x=\begin{cases}
x+B^1_t, & \quad t \leq \tau_{-(y+x)/2}, \\
-(y+x)/2 + B^2_{t-\tau_{-(y+x)/2}}, & \quad t \geq \tau_{-(y+x)/2},
\end{cases}
\end{equation*}
and
\begin{equation*}
X_t^y=\begin{cases}
y+ B^1_t, & \quad t \leq \tau_{-(y+x)/2}, \\
(y+x)/2 - B^2_{t-\tau_{-(y+x)/2}}, & \quad t \geq \tau_{-(y+x)/2},
\end{cases}
\end{equation*}
are Brownian motions starting in~$x$ and~$y$, respectively. Until time~$\tau_{-(y+x)/2}$, the process~$X_t^y$ spends more time outside~$(-1,1)$ than~$X_t^x$. Afterwards, both processes spend the same amount of time outside~$(-1,1)$. Thus it holds
\begin{align*}
\int_0^T \1_{\{\abs X_t^x\abs \ge 1\}} \D t \leq \int_0^T \1_{\{\abs X_t^y\abs \ge 1\}} \D t
\end{align*}
proving the claim.
\end{proof}

It essentially remains to combine Proposition~\ref{prop:asympSnu} and Lemma~\ref{l:quotientPyPnu}:
\begin{proof}[Proof of Proposition~\ref{prop:asymp}]
Let~$y \in (-1,1)$. Lemma~\ref{l:quotientPyPnu} and Proposition~\ref{prop:asympSnu} imply
\begin{align*}
{}&\p_y(\Gamma_T \in (0,s])\\
={}& \frac{\p_y(\Gamma_T \in (0,s])}{\p_{\nu}(\Gamma_T \in (0,s])} \cdot \e^{-\frac{\pi^2}{8} T} S_\nu(T)\\
\sim{}& \frac{4}{\pi} \cos\left(\frac{\pi y}{2}\right) \cdot \frac{ 2^{\frac{7}{6}}}{\sqrt{3} \pi^{\frac{7}{6}} s^{\frac{1}{6}} T^{\frac{1}{3}}} \exp\left( -\frac{\pi^2}{8} T + \frac{3}{2^{\frac{7}{3}}} \pi^{\frac{4}{3}} s^{\frac{1}{3}} T^{\frac{2}{3}} - \frac{3}{2^{\frac{5}{3}}} \pi^{\frac{2}{3}} s^{\frac{2}{3}} T^{\frac{1}{3}} +\frac{\pi^2+12}{24}s \right), \quad T \to \infty.
\end{align*}
Comparing this with equation~\eqref{eqn:asympGamma0}, we see that $$\p_y(\Gamma_T \le s)= \p_y(\Gamma_T = 0) + \p_y(\Gamma_T \in (0,s])$$ must have the same asymptotic behavior as~$T \to \infty$ proving the first claim.\\
Now let~$x \in \R \setminus (-1,1)$. Using Lemma~\ref{l:comparison} and the case already settled, we get
\begin{align*}
   \limsup_{T\rightarrow \infty}  \frac{\p_x(\Gamma_T \le s)}{\p_0(\Gamma_T \le s)} \leq \limsup_{T\rightarrow \infty} \frac{\p_{y}(\Gamma_T \le s)}{\p_0(\Gamma_T \le s)} = \cos\left(\frac{\pi y}{2} \right), \quad y \in (-1,1),
\end{align*}
which, taking~$y \nearrow 1$, implies 
\begin{align*}
   \limsup_{T\rightarrow \infty}  \frac{\p_x(\Gamma_T \le s)}{\p_0(\Gamma_T \le s)} = 0,
\end{align*}
proving the second claim.
\end{proof}

\section{Proofs of Theorem~\ref{thm:main} and Corollary \ref{cor:outside}} \label{sec:proofofmaintheorem}
Let us now prove our main result. To this end, let~$y \in (-1,1)$. It is a classical result (see, e.g., Example~1 in~\cite{Pin85}) that $$\p_y(B \in \cdot\,\,\abs\, \Gamma_T = 0) = \p_y\left(B \in \cdot\,\,\big\abs \, \abs B_t\abs \lne 1 \text{ for all } t \in [0,T]\right)$$ converges weakly to a probability measure~$\Q_y$ on~$\mathcal{C}([0,\infty),\mathbb{R})$ as~$T \to \infty$ and that the limiting process~$(Y_t)_{t \ge 0} \sim \Q_y$ satisfies the SDE~\eqref{eqn:sde}. Regarding the weak convergence of $\p_y(B \in \cdot\,\,\abs\, \Gamma_T \le s)$ as~$T \to \infty$, we follow the classical approach: We prove convergence of finite-dimensional distributions and tightness.

\begin{lem}
As~$T \to \infty$, the probability measures $\p_y(B \in \cdot \,\, \abs\, \Gamma_T \le s)$ converge to~$\Q_y$ in finite-dimensional distributions.
\end{lem}
\begin{proof}
Let~$0 \lne t_1 \lne \dots \lne t_d$ and~$C_1,\dots,C_d \in \BB(\R)$. Moreover, let~$\tau$ be the first exit time of~$B$ from~$(-1,1)$. We start by observing
\begin{align*}
{}&\p_y(B_{t_1} \in C_1, \dots, B_{t_d} \in C_d \,\abs\, \Gamma_T\le s)\\
={}&\p_y(B_{t_1} \in C_1,\dots, B_{t_d} \in C_d, \tau \gne t_d\,\abs\, \Gamma_T\le s) + \p_y(B_{t_1} \in C_1, \dots, B_{t_d} \in C_d, \tau \le t_d\,\abs\, \Gamma_T\le s), \quad T \ge 0.
\end{align*}
The Markov property implies
\begin{align}
\p_y(B_{t_1} \in C_1, \dots, B_{t_d} \in C_d, \tau \le t_d\,\abs\, \Gamma_T\le s) \le{}& \p_y(\tau \le t_d\,\abs\, \Gamma_T \le s)\notag\\
\le{}& \frac{\p_1(\Gamma_{T-t_d} \le s)}{\p_y(\Gamma_{T-t_d} \le s)}\cdot \frac{\p_y(\Gamma_{T-t_d} \le s)}{\p_y(\Gamma_T \le s)}, \label{eqn:quot3} \quad T \ge 0.
\end{align}
The first part of Proposition~\ref{prop:asymp} yields
\begin{align}
\lim_{T \to \infty} \frac{\p_{x_2}(\Gamma_{T-t} \le s)}{\p_{x_1}(\Gamma_T \le s)}
= \frac{\cos\left(\frac{\pi x_2}{2}\right)}{\cos\left(\frac{\pi x_1}{2}\right)} \e^{\frac{\pi^2}{8} t}, \quad x_1,x_2 \in (-1,1),\ t \ge 0.
\label{eqn:quot2}
\end{align}
This and the second part of Proposition~\ref{prop:asymp} imply that the term in~\eqref{eqn:quot3} converges to~$0$ as~$T \to \infty$. Now let $$(p_t: (-1,1) \times (-1,1) \to [0,\infty))_{t \gne 0}$$ be the family of (non-probability) transition densities of~$B$ absorbed in~$\{-1,1\}$. Further, let~$x_0:=y$ and~$t_0:=0$. Using the Markov property and the dominated convergence theorem together with~\eqref{eqn:quot2}, we obtain
\begin{align*}
{}&\lim_{T \to \infty} \p_y(B_{t_1} \in C_1, \dots, B_{t_d} \in C_d, \tau \gne t_d\,\abs\, \Gamma_T\le s)\\
={}& \lim_{T \to \infty} \int_{(-1,1)^d} \1_{C_1\times \dots \times C_d}(\boldsymbol{x}) \prod_{i=1}^d p_{t_i-t_{i-1}}(x_{i-1},x_i) \frac{\p_{x_d}(\Gamma_{T-t_d} \le s)}{\p_y(\Gamma_T \le s)} \D \boldsymbol{x}\\
={}& \int_{(-1,1)^d} \1_{C_1\times \dots \times C_d}(\boldsymbol{x}) \prod_{i=1}^d p_{t_i-t_{i-1}}(x_{i-1},x_i) \frac{\cos\left(\frac{\pi x_d}{2}\right)}{\cos\left(\frac{\pi y}{2}\right)}\e^{\frac{\pi^2}{8} t_d} \D \boldsymbol{x}\\
={}& \int_{(-1,1)^d} \1_{C_1\times \dots \times C_d}(\boldsymbol{x}) \prod_{i=1}^d \left(p_{t_i-t_{i-1}}(x_{i-1},x_i) \frac{\cos\left(\frac{\pi x_i}{2}\right)}{\cos\left(\frac{\pi x_{i-1}}{2}\right)}\e^{\frac{\pi^2}{8} (t_i-t_{i-1})}\right) \D \boldsymbol{x},
\end{align*}
where~$x_1,\dots,x_d$ are the components of~$\boldsymbol{x}$. According to Theorem~3.1(ii) of~\cite{CV16} and equation~\eqref{eqn:asympGamma0}, the family~$(\tilde{p}_t)_{t \gne 0}$ of transition densities of the process with law~$\Q_y$ is given by 
\begin{align*}
\tilde{p}_t(x_1,x_2)
={}& \e^{\frac{\pi^2}{8}t} \lim_{T \to \infty} \frac{\p_{x_2}(\Gamma_T = 0)}{\p_{x_1}(\Gamma_T = 0)} p_t(x_1,x_2)\\
={}& \e^{\frac{\pi^2}{8}t}  \frac{\cos\left(\frac{\pi x_2}{2}\right)}{\cos\left(\frac{\pi x_1}{2}\right)} p_t(x_1,x_2), \quad x_1,x_2 \in (-1,1),\ t \gne 0,
\end{align*}
completing the proof.
\end{proof}

To complete the proof of Theorem~\ref{thm:main}, it now suffices to prove tightness of the family of conditional laws with the help of the Kolmogorov-Chentsov theorem.
\begin{lem}
The family~$(\mathbb{P}_y(B \in \cdot\,\,\abs\, \Gamma_T\leq s))_{T\ge 0}$ of probability measures on the Wiener space $\mathcal{C}([0,\infty),\mathbb{R})$ is tight.
\end{lem}
\begin{proof}
Let~$t_0 \gne 0$ and~$t_1,t_2 \in [0,t_0]$ with~$t_1 \le t_2$. Using the Markov property and Lemma~\ref{l:comparison}, we get
\begin{align*}
\E_y \left[\abs B_{t_2}-B_{t_1}\abs^4\,\big\abs\, \Gamma_T\le s\right]
{}&= \frac{1}{\p_y(\Gamma_T\le s)} \E_y\left[\abs B_{t_2}-B_{t_1}\abs^4 \1_{\{\Gamma_T \le s\}} \right] \\
{}&\le \frac{1}{\p_y(\Gamma_T\le s)} \E_y\left[\abs B_{t_2}-B_{t_1}\abs^4\1_{\{\Gamma_T-\Gamma_{t_2} \le s\}} \right] \\
{}& =  \frac{1}{\p_y(\Gamma_T\le s)} \E_y\left[\abs B_{t_2}-B_{t_1}\abs^4 \p_{B_{t_2}}(\Gamma_{T-t_2}\le s) \right] \\
{}&\le\frac{\p_0(\Gamma_{T-t_2}\le s)}{\p_y(\Gamma_T\le s)}\E_y\left[\abs B_{t_2}-B_{t_1}\abs^4\right]\\
{}&\le\frac{\p_0(\Gamma_{T-t_0}\le s)}{\p_y(\Gamma_T\le s)} 3 \abs t_2-t_1\abs^2, \quad T \ge t_0.
\end{align*}
Since Proposition~\ref{prop:asymp} implies $\limsup_{T\rightarrow \infty}\frac{\p_0(\Gamma_{T-t_0}\le s)}{\p_y(\Gamma_T\le s)} \lne \infty$ and the probabilities are continuous in~$T$, there must be a constant~$C \gne  0$ with $$\E_y \left[\abs B_{t_2}-B_{t_1}\abs^4\,\big\abs\, \Gamma_T\le s\right] \le C \abs t_2-t_1\abs^2, \quad t_1,t_2 \in [0,t_0],\ T \ge 0.$$
Now let~$\eps\gne 0$ and~$\gamma \in (0, \frac{1}{2})$. As a consequence of the theorem of Kolmogorov-Chentsov (see, e.g., Theorem~3.4.16 in~\cite{Str93} for an applicable version), there exists a constant~$K \gne  0$ (independent of~$T$) with
\begin{align*}
\p_y\left(\abs X_{t_2}-X_{t_1}\abs \le K\abs t_2-t_1\abs^{\gamma} \text{ for all }t_1,t_2 \in [0,t_0]\,\big\abs\, \Gamma_T\le s\right) \ge 1-\eps, \quad T \ge 0.
\end{align*}
Moreover, the set 
\begin{align*}
\left\{X \in \CC([0,t_0]):\abs X_{t_2}-X_{t_1} \abs\le K \abs t_2-t_1\abs^{\gamma} \text{ for all }t_1,t_2 \in [0,t_0] \right\}
\end{align*}
is compact in the space~$\mathcal{C}([0,t_0],\mathbb{R})$. Hence we can apply Prohorov's theorem to conclude tightness on~$\mathcal{C}([0,t_0],\mathbb{R})$. Corollary~5 in~\cite{Whi70} yields tightness on~$\mathcal{C}([0,\infty),\mathbb{R})$.
\end{proof}

Finally, we deduce Corollary \ref{cor:outside} from Proposition \ref{prop:asymp}:

\begin{proof}[Proof of Corollary \ref{cor:outside}]
W.l.o.g., let~$\eps \in (0,s)$ and $T \gne \eps$. By symmetry, we may assume~$y \gne 1$. Noting that the mapping $t \mapsto\{\Gamma_{T+t} \le s+t\}$ is (weakly) monotonically increasing on~$[-T,\infty)$ w.r.t.~inclusion, we, on the one hand, obtain
\begin{align*}
\p_y(\tau \ge \eps, \Gamma_T \le s)
= \int_\eps^s \p_1(\Gamma_{T-t} \le s-t) \p_y(\tau \in \D t)
\le \p_1(\Gamma_{T-\eps} \le s-\eps)\p_y(\tau \in [\eps,s]).
\end{align*}
Now let $\tau_0:= \inf\{t \ge 0: B_t =0\}$. On the other hand, we then get
\begin{align*}
\p_y(\Gamma_T \le s)
\ge{}& \p_y(\tau_0 \le \eps,\Gamma_T \le s)\\
\ge{}& \int_0^\eps \p_0(\Gamma_{T-t} \le s-t) \p_y(\tau_0 \in \D t) \ge \p_0(\Gamma_{T-\eps} \le s-\eps) \p_y(\tau_0 \in (0,\eps)).
\end{align*}
Noting $\p_y(\tau_0 \in (0,\eps))\gne 0$, the second part of Proposition \ref{prop:asymp} yields the claim.
\end{proof}

\section*{Acknowledgements}
The second named author would like to thank Mladen Savov and Achim W\"ubker for stimulating discussion on the topic of this work.

%\bibliographystyle{alpha}
%\bibliography{LiteratureAKS23-short}

\end{document}